\documentclass[12pt]{scrartcl}
\addtokomafont{disposition}{\rmfamily}

\usepackage[utf8]{inputenc}
\usepackage[T1]{fontenc}
\usepackage{lmodern}
\usepackage{amsmath,amsfonts,amssymb,amsthm,amscd}
\usepackage{xcolor}
\usepackage{microtype}
\usepackage{csquotes}
\MakeOuterQuote{"}

\usepackage[style=alphabetic,backend=biber,sorting=nyt,giveninits=true]{biblatex}
\usepackage[colorlinks=true, breaklinks=true, urlcolor= black, linkcolor= black, citecolor= black, 
bookmarksopen=true,linktocpage=true,plainpages=false,pdfpagelabels]{hyperref}
\renewbibmacro{in:}{}
\bibliography{biblio}

\usepackage[capitalize]{cleveref}

\title{On groups definable in geometric fields with generic derivations}
\date{\today}
\author{Anand Pillay\thanks{University of Notre-Dame, supported by NSF grants DMS-2054271 and DMS-2502292} \and Fran\c coise Point\thanks{Universit\' e de Mons.} \and Silvain Rideau-Kikuchi\thanks{CNRS, \'Ecole normale sup\' erieure, partially supported
by GeoMod AAPG2019 (ANR-DFG), Geometric and Combinatorial Configurations in
Model Theory.}}

\newtheorem{Theorem}{Theorem}[section]
\newtheorem*{Theorem*}{Theorem}
\newtheorem{Proposition}[Theorem]{Proposition}
\newtheorem{Definition}[Theorem]{Definition}

\newtheorem{Lemma}[Theorem]{Lemma}
\newtheorem{Claim}[Theorem]{Claim}
\newtheorem{Corollary}[Theorem]{Corollary}

\newcommand{\N}{\mathbb N}

\newcommand{\dcl}{\mathrm{dcl}}
\newcommand\acl{\mathrm{acl}}
\newcommand\Td{T_\Delta}
\newcommand\Ld{L_\Delta}
\newcommand\prol[1]{\nabla_{\!#1}}
\newcommand\dcld{\dcl_\Delta}
\newcommand\acld{\acl_\Delta}
\newcommand{\alg}[1]{#1^{\mathrm{alg}}}
\newcommand\tensor{\otimes}
\newcommand{\DCF}{\mathrm{DCF}}
\newcommand{\tp}{\mathrm{tp}}
\newcommand\prolim{\varprojlim}
\newcommand{\drg}{{\Delta\mathrm{rg}}}
\newcommand\Ldrg{L_{\drg}}
\newcommand\restr[2]{#1\left|_{#2}\right.}
\newcommand\supsel{\succcurlyeq}
\renewcommand\mid{:}
\renewcommand\projlim{\varprojlim}
\newcommand\Stab{\mathrm{Stab}}

\begin{document}
\maketitle

\begin{abstract}
We study groups definable in existentially closed geometric fields
with commuting derivations. Our main result is that such a group can
be definably embedded in a group interpretable in the underlying
geometric field. Compared to earlier work of the first two authors
together with K. Peterzil, the novelty is that we also deal with
infinite dimensional groups.
\end{abstract}

\section*{Introduction}

This paper is about comparing groups definable in a given structure with groups
definable in a reduct. These latter groups are usually better understood,
providing in return many insights on the original groups.

The case of interest here is to compare groups definable in a field with a
derivation (and potentially additional structure), to groups definable without
the derivation. In the case of a differentially closed field\footnote{With no
further structure.}, the second author proved in \cite{Pil-GpDCF} that a
differential algebraic group differentially algebraically embeds in an algebraic
group (answering questions of Kolchin). The methods were stability-theoretic and
reminiscent of Weil's construction of an algebraic group out of a pregroup
\cite{Wei-PreGp}.

Our main goal here is to consider a generalization of this result to an unstable
context. We will be working with enriched geometric fields of characteristic
$0$. Following \cite[Definition~2.9]{HruPil}, but allowing additional structure,
we say that a (complete) theory \(T\) of fields with additional structures is
\emph{a geometric theory of enriched fields} if:
\begin{itemize}
\item In models of \(T\), model theoretic algebraic closure coincides with
relative field theoretic algebraic closure (over \(\dcl(\emptyset)\)).
\item Models of \(T\) are perfect and \(T\) eliminates the \(\exists^\infty\)
quantifier.
\end{itemize}

As noted by multiple authors, \cite{For,JohYe-Geom}, the first hypothesis
implies the second. Such a theory \(T\) is also said to be \emph{algebraically
bounded}.

Examples include real closed fields, \(p\)-adically closed fields, more
generally characteristic zero henselian valued fields, bounded pseudo
algebraically closed fields (such as pseudofinite fields), open theories of
topological fields \cite{CKPoi}, bounded perfect pseudo \(T\)-closed fields
\cite{MonRK-PTC} and curve excluding fields \cite{JohYe-CurvEx}, among others.

Recall that a geometric theory is a complete theory in an arbitrary language
such that in models of $T$ algebraic closure is a pregeometry and \(T\)
eliminate the \(\exists^\infty\) quantifier. So a geometric theory of enriched
fields is a geometric theory, but the key point is that algebraic closure
coincides with (relative) field-theoretic algebraic closure. 

So let us fix some geometric theory \(T\) of enriched fields in some language
\(L\). At the cost of Morleyizing, we can always assume that \(T\) eliminates
quantifiers in a relational expansion of the ring language. Also fix some
integer \(\ell\) and \(\Delta = \{\partial_1,\ldots,\partial_\ell\}\). Consider
the theory of models of \(T\) with \(\ell\) commuting derivations in the
language \(\Ld = L\cup\Delta\). By \cite[Theorem~4.1]{ForTer}, this theory
admits a model completion that we will denote \(\Td\).

In this context, we prove the following.

\begin{Theorem*}[\cref{emb Gp}]
Let \(K \models \Td\) and let \(\Gamma\) be a group which is \(\Ld\)-definable
in \(K\). Then there is a group \(G\) which is \(L\)-interpretable in (the
reduct to \(L\) of) \(K\) and an \(\Ld\)-definable embedding \(\Gamma\to G\).
\end{Theorem*}

Moreover, if \(\Gamma\) is definable over some \(A\subseteq K\), then the group
\(G\) and the embedding can also be chosen over \(A\).

In \cite{PetPilPoi} the first two authors and K. Peterzil proved the
finite-dimensional case (see \cite[Definition~2.2]{PetPilPoi}) of our main result
--- first considering $K \models RCF$, then generalizing. Assuming that
\(\Gamma\) is finite dimensional, one recovers a generically given
\(L\)-definable group, and proceeds from there. In the case of possibly
infinite-dimensional $\Gamma$ one will obtain some kind of generically given
\(L\)-definable group but living on infinite tuples, and there are additional
technical complications.  It turns out that general results from
\cite{HruRK-MetaGp} are in a sense tailor-made to handle such situations, so we
will appeal to them. Possibly, the methods of \cite{KowPil} giving another
account of \cite{Pil-GpDCF} would also adapt to the present setting.

On the face of it, the main theorem for finite-dimensional definable groups
follows from the theorem for arbitrary definable groups. However in
\cite{PetPilPoi}, the group \(G\) can be chosen to be \(L\)-definable and not
only \(L\)-interpretable --- no quotient is required.

Note also that, in \cite{PetPilPoi}, the finite-dimensional case was also
considered in other contexts, namely $o$-minimal expansions of real closed
fields, which fall outside the scope of the present paper. Also, in
\cite{PetPilPoi-2}, Buium's notion of an "algebraic D-group" is adapted from
the context of algebraically closed fields with a derivation to models of \(T\)
with a derivation and finite-dimensional \(\Ld\)-definable groups are shown to
be precisely groups of "sharp points" of D-groups.

We would like to thank K. Peterzil for discussions on the strategy, ideas and
content of the present paper. We also thank the anonymous referee for their many
helpful comments.

\section{Geometric fields with generic derivations}

Let us start by recalling some facts about geometric fields.

\begin{Definition}
Let \(K\models T\) and let \(X\subseteq K^n\) be definable (with parameters).
The dimension \(\dim(X)\) of \(X\) is the dimension (in the sense of algebraic
geometry) of the Zariski closure of \(X\) in affine \(n\)-space.
\end{Definition}

If \(X\) is definable over some subfield \(A\subseteq K\) and if \(K\) is sufficiently
saturated, \(\dim(X)\) is the maximal transcendence degree of some \(c\in X(K)\)
over \(A\).

In a geometric theory, dimension is definable in the sense that for all
definable \(X\subseteq Y\times Z\), for every \(n\geq 0\), \(\{z\in Z \mid
\dim(X_z) = n\}\) is definable, where \(X_z = \{y\in Y\mid (y,z) \in X\}\) is
the fiber of \(X\) at \(z\).

Let us now consider the properties of models of \(\Td\). Let \(K,M \models \Td\)
and let \(A\subseteq K\). We say that an injective ring embedding \(f : A\to M\)
is an \(L\)-elementary embedding if it preserves \(L\)-types (between the
reducts of \(K\) and \(M\) to \(L\)) and that it is differential if it preserves
the derivations. If we want to specify that we work in the reduct to \(L\), we
will indicate it with a subscript \(_L\), like in the notation \(\dcl_L\) for
the \(L\)-definable closure. Notation with a subscript \(_\drg\) will refer to
the differential field structure and notions with a subscript \(_\Delta\), or
no subscript, will refer to the full \(\Ld\)-structure.

\begin{Lemma}
\label{Td embed}
Let \(K,M\models \Td\), let \(A\leq_\drg K\) be a differential subfield and
let \(f : A \to M\) be a differential \(L\)-elementary embedding. Then there
exists an \(\Ld\)-elementary extension \(M'\succcurlyeq M\) and a differential
\(L\)-elementary embedding \(g : K\to M'\) extending \(f\) such that
\(g(K)\) is algebraically independent of \(M\) over \(g(A)\).
\end{Lemma}

\begin{proof}
Let \(A_0 = \alg{A}\cap K\) be the relative field algebraic closure of \(A\) in
\(K\). Since \(f\) is \(L\)-elementary, and $L$ contains the ring language, 
we can extend \(f\) to \(\overline{f} :
A_0 \to M\) as an \(L\)-elementary embedding. Moreover, since \(\Delta\)
extends uniquely to \(\alg{A}\), the embedding \(\overline{f}\) is also a differential field
embedding.

Fix some tuple \(a \in K^m\) and let \(V\) be its (geometrically integral) locus
over \(A_0\). For any \(X\in \tp_L(a/A_0)\) (namely $X$ is the solution set of a formula
in the type) contained in \(V\), let \(W\) be the
Zariski closure of \(X\) over \(K\). Then \(W\subseteq V\), it contains \(a\)
and, by invariance, it is defined over \(\dcl_L(A_0) = A_0\) in the language $L$. So \(W = V\). 
By compactness, $tp_{L}(a/A_{0}) \cup \{\neg\phi(x):\phi$ defining a proper Zariski closed subset of $V$, defined over $K\}$
is consistent. 
It follows that we can extend \(\overline{f}\) to an
\(L\)-elementary \(g_L : K \to M' \succcurlyeq_L M\) such that \(g_L(K)\) is
algebraically independent of \(M\) over \(g_L(A_0)\).

Since \(A_0\leq K\) is a regular extension, it follows that the compositum
\(g_L(K)M\) is isomorphic to (the fraction field of) \(K\tensor_{A_0} M\) which
can therefore be made into a substructure of a model of \(T\) extending the
\(L\)-structure on both \(K\) and \(M\). The derivations on \(K\) and \(M\) also
extend uniquely to \(K\tensor_{A_0} M\) and the resulting derivations commute.
Since \(M\leq K\tensor_{A_0} M\) is \(\Ld\)-existentially closed,
\(K\tensor_{A_0} M\) can be embedded into some \(\Ld\)-elementary extension
\(M_1 \succcurlyeq M\), concluding the proof.
\end{proof}

Let \(\Theta\) denote the commutative monoid generated by \(\Delta\). Its
elements are of the form \(\theta = \partial_1^{e_1}
\ldots\partial_\ell^{e_\ell}\) for all integers \(e_1,\ldots e_\ell \geq 0\).
For such a \(\theta\), we define \(|\theta| = \sum_{i\leq \ell} e_i\). We order
the elements of \(\Theta\) by lexicographic order on \(|\theta|\) and the
components of \(\theta\).

For any tuple \(a\) in a differential field \((K,\Delta)\) and for any integer
\(n\geq 0\), we write \(\prol{n}(a) = (\theta a_i)_{i,|\theta|\leq n}\). We also
write \(\prol{\omega}(a) = (\theta a)_{i,\theta}\).

From \cref{Td embed}, we immediately recover a strong form of quantifier
elimination\footnote{This is implicit in \cite{ForTer}, but it is only stated
explicitly for a single derivation.}. Let \(K\models\Td\) and let \(A\subseteq K\).

\begin{Corollary}
\label{EQ Td}
\begin{enumerate}
\item Differential $L$-elementary maps between differential subfields of models
of $T_{\Delta}$ are $L_{\Delta}$-elementary.
\item If \(X\) is \(\Ld\)-definable over \(A\) in \(K\), there exists an integer \(n\)
and a set \(Y\) which is \(L\)-definable over \(A\) such that \(x\in X\) if and
only if \(\prol{n}(x) \in Y\).
\end{enumerate}
\end{Corollary}

\begin{proof}
Let \(K, M\models \Td\) be sufficiently saturated, let \(A\leq_\drg K\) be a
small differential subfield and \(f : A \to M\) be a differential
\(L\)-elementary embedding. Let \(B\leq K\) be small and contain \(A\). By
\cref{Td embed}, we may extend \(f\) to a differential \(L\)-elementary
embedding \(B \to M\) --- enlarging \(B\), we may  assume that it is an
\(\Ld\)-elementary substructure of \(K\). In other words, differential
\(L\)-elementary isomorphisms between small differential subfields have the
back-and-forth property and hence are \(\Ld\)-elementary. The first statement is
proved.

Now, it follows that, for every tuple \(a\) and \(b\) in \(K\), if
\(\tp_{L}(\prol{\omega}(a)) = \tp_{L}(\prol{\omega}(b))\) then \(\tp(a) =
\tp(b)\). The second statement  follows by compactness.
\end{proof}

We can also immediately characterize algebraic and definable closure in models
of \(\Td\).

\begin{Corollary}
\label{dcl Td}
\begin{enumerate}
\item The \(\Ld\)-algebraic closure \(\acld(A)\) of \(A\) is the relative field
algebraic closure in \(K\) of the differential field generated by \(A\).
\item The \(\Ld\)-definable closure \(\dcld(A)\) of \(A\) is the \(L\)-definable
closure of the differential field generated by \(A\).
\item Let  \(f : X\to Y\) be \(\Ld\)-definable over \(A\). There exists \(F : Z
\to W\) which is \(L\)-definable over \(A\) and an integer \(n\geq 0\) such that
\(\prol{n}(X)\subseteq Z\) and, for all \(x\in X\), \(f(x) = F(\prol{n}(x))\).
\end{enumerate}
\end{Corollary}

\begin{proof}

Let \(A  = \alg{A}\cap K \leq_\drg K\) be a relatively algebraically closed
differential subfield. By \cref{Td embed}, there exists a differential
\(L\)-elementary embedding \(g : K \to K' \succcurlyeq K\) extending the
identity on \(A\) and such that \(g(K)\) and \(K\) are algebraically independent
over \(A\). Since \(A\leq K\) is regular, \(K\) and \(g(K)\) are linearly
disjoint over \(A\). Moreover, by \cref{EQ Td}, \(g\) is \(\Ld\)-elementary and
thus \(g(K)\preccurlyeq K'\). It follows that \(\acld(A) \subseteq K\cap
g(K) = A\), concluding the proof of the first item.

Now, let \(A = \dcl_L(A) \leq_\drg K\) be a \(\dcl_L\)-closed differential
subfield of \(K\). By the first item, we have \(\dcld(A)\subseteq \acld(A)
\subseteq \alg{A}\). Consider \(a \in \alg{A}\cap K\setminus A\). Then, since
\(a\not\in\dcl_L(A)\), it has at least one other \(L\)-conjugate \(a'\in M\)
over \(A\). Since \(a,a'\in\alg{A}\), any \(L\)-elementary embedding sending
\(a\) to \(a'\) is also a differential embedding and hence, by \cref{EQ Td}, it
is \(\Ld\)-elementary. So \(a \not\in \dcld(A)\) and \(\dcld(A)\subseteq A\),
proving the second item.

Finally, let \(f : X\to Y\) be \(\Ld\)-definable over \(A\). For every \(x\in
X\), by the previous item, we have \(f(x)\in \dcl_L(A \prol{\omega}(x))\). By
compactness, it follows that there are finitely many maps \(F_i\) which are
\(L\)-definable over \(A\) such that for every \(x\in X\), there exists an \(i\)
such that \(f(x) = F_i(\prol{\omega}(x))\). Let \(X_i = \{x\in X\mid f(x) =
F_i(\prol{\omega}(x))\}\) and, by \cref{EQ Td}, let \(Z_i\) be \(L\)-definable
over \(A\) such that, for some sufficiently large \(n\), \(x\in X_i\) if and
only if \(\prol{n}(x) \in Z_i\). We may assume that the \(Z_i\) are disjoint. We
define \(F\) on \(Z = \bigcup_i Z_i\) by \(F(z) = F_i(z)\) if \(z\in Z_i\).
Then, for any \(x\in X\), we have \(f(x) = F(\prol{n}(x))\).
\end{proof}

Let us conclude this section with a purely differential statement which is
implicit in the proof of the existence of the Kolchin polynomial.

\begin{Lemma}
\label{pure tr}
Let \((K,\Delta)\) be a differential field (with finitely many commuting
derivations). Let \(K\leq_\drg K\langle a_1,\ldots,a_m\rangle_\Delta\) be a
finitely generated differential field extension. Let $a = (a_1,..,a_m)$. Then there exists an \(n\neq
0\) such that the extension \(K(\prol{n}(a))\leq K(\prol{\omega}(a))\) is purely
transcendental.
\end{Lemma}

\begin{proof}
We order the set of \(\theta a_i\), for all \(\theta\in\Theta\) and \(i\leq n\)
by lexical order first on \(\theta\) and then on \(i\) --- this is a well order
isomorphic to \(\omega\). Let \(E\) be the set of \(\theta a_i\) such that
\(\theta a_i \in \alg{K(\theta' a_j \mid \theta' a_j < \theta a_i)}\). For any
\(\theta a_i \in E\) and any \(\eta\in\Theta \setminus\{1\}\), we have \(\eta
\theta a_i \in K(\theta' a_j \mid \theta' a_j \leq \theta a_i)\).

Let \(E_0\) be the set of \(\theta a_i\in E\) which are not proper derivatives
of any element of \(E\). Identifying \(\{\theta a_i\mid \theta \in\Theta \text{
and } i\leq m\}\) naturally with a subset of \(\omega^\ell\times \{1,\ldots,
m\}\), if \(E_0\) is infinite, it contains a strictly increasing sequence for
the product order --- indeed starting with any sequence of pairwise distinct
elements in \(E\), one can iteratively extract subsequences to make each
projection increasing. So there is some \(i\leq m\) and some
\(\theta,\eta\in\Theta\) such that \(\theta a_i\) and \(\eta\theta a_i\) are in
\(E_0\). This contradicts the previous paragraph so \(E_0\) is finite.

Also, we have \(E = \Theta E_0\). Let \(n = \max_{\theta\in E_0} |\theta|\).
Then \(K(\prol{n}(a))\leq K(\prol{\omega}(a))\) is purely transcendental.
Indeed, for any \(\theta a_i\) with \(|\theta| > n\), either \(\theta a_i \in
E\) in which case \(\theta a_i \in K(\theta' a_j \mid \theta' a_j < \theta
a_i)\) or \(\theta a_i \not\in E\) in which case \(\theta a_i\) is
transcendental over \(K(\theta' a_j \mid \theta' a_j < \theta a_i)\).
\end{proof}

\section{Generic points}

Let \(K \models\Td\) be sufficiently saturated and homogeneous. We also fix an
elementary extension \(M\supsel K\) which is \(|K|^+\)-saturated in which to
realize (partial) types over \(K\). Let \(\Ldrg\) denote the language of differential rings. 

Recall that the \(\Ldrg\)-theory $\DCF_{\Delta, 0}$ of differentially closed
fields of characteristic $0$ with respect to the commuting derivations in
$\Delta$, is $\omega$-stable and has quantifier elimination. We let ${\mathcal
U}$ be a big model of $DCF_{\Delta,0}$ containing $M$. So complete quantifier
free types in \(\Ldrg\) over subsets of ${\mathcal U}$ correspond to complete
types in the sense of $\DCF_{\Delta, 0}$, so, as such, have an ordinal valued
Morley rank.  For a tuple $a$ from $M$ we will define the Morley rank of the
quantifier-free \(\Ldrg\) -type of $a$ over $K$ to be its Morley rank in
$\DCF_{\Delta,0}$.

Let $X$ be \(\Ld\)-definable in \(K\).  Let $S$ be the set of complete
quantifier-free types over $K$ (in the sense of  $\DCF_{\Delta, 0}$) which are
finitely satisfiable in $K$ by elements of $X$, equivalently realized in $M$ by
an element of $X$. Then by compactness (in $T_{\Delta}$), the set $S$ is a
closed subset of the space of quantifier-free complete types over $K$, in the
sense of $\DCF_{\Delta, 0}$.  So $S$ corresponds to a partial (quantifier-free)
type $\Sigma$ over $K$ in the sense of $\DCF_{\Delta, 0}$. Let \(\alpha\) be the
Morley rank of \(\Sigma\). By properties of Morley rank, we have:



\begin{Lemma}
$\Sigma$ extends to finitely many complete quantifier-free types over $K$ of Morley rank $\alpha$, say $p_{1},..,p_{k}$. 
\end{Lemma}

We expect that $p_{1},..,p_{k}$ are also the types in the space $S$ of maximal
Cantor-Bendixon rank, but we will not need to know this.




\begin{Definition}
We say that \(a\in X\) --- in \(M\) --- is generic (over \(K\)) if it realizes
one of the \(p_i\) --- in other words, \(a\) is generic in \(X\) if \(a\) has
maximum Morley rank in \(X\) over \(K\).
\end{Definition}

\begin{Lemma}
\label{inj generic}
Let \(f : X \to X\) be  \(\Ld\)-definable over \(K\) and injective. Let \(a\) be generic
in \(X\), then \(f(a)\) is also generic.
\end{Lemma}

\begin{proof}
By hypothesis, \(\dcld(Ka) = \dcld(K f(a))\). It follows from \cref{dcl Td}.1,
that the field algebraic closure of the differential fields generated by \(a\)
and \(f(a)\) over \(K\) are identical, and hence that \(a\) and \(f(a)\) have
the same Morley rank over \(K\). So \(a\) is generic in \(X\) if and only if
\(f(a)\) is.
\end{proof}

Let \(\Sigma_X(x_\omega)\) be the common \(L\)-type over \(K\) of
\(\prol{\omega}(a)\), where \(a\in X\) is generic over \(K\). By quantifier
elimination (\cref{EQ Td}), \(\Sigma_X(\prol{\omega}(x))\) is the partial type of
generics in \(X\).

\begin{Lemma}
\label{gen def}
Let \(A = \dcld(A)\subseteq K\) be such that \(X\) is definable over \(A\). Then
the partial type \(\Sigma_X\) is \(L\)-definable over \(A\) --- that is, for every
formula \(\phi(x_\omega,y)\), the set of \(a\in K^y\) such that
\(\Sigma_X(x_\omega)\models\phi(x_\omega,a)\) is \(L\)-definable over \(A\).
\end{Lemma}

\begin{proof}
We write $\dim(a/K)$ for the transcendence degree of \(K(a)\) over \(K\).

For every \(n\geq 0\), let \(W_{i,n}\) be the Zariski locus of \(\prol{n}(a)\)
over \(K\) for any (equivalently all) \(a\models p_i\). Then \(a\models p_i\) if
and only if, for every \(n\), \(\prol{n}(a) \in W_{i,n}\) and
\(\dim(\prol{n}(a)/K) = \dim(W_{i,n})\).

Assume that \(n\) is sufficiently large so
that \(x\in X\) if and only if \(\prol{n}(x)\in Y\), for some \(Y\) which is
\(L\)-definable over \(A\) (by Corollary 1.3 (ii)). Also assume that \(n\) is sufficiently large so
that, by \cref{pure tr}, for every \(i\) and every \(a\models p_i\), the
extension \(K(\prol{n}(a))\leq K(\prol{\omega}(a))\) is purely transcendental.
For future reference, let us fix some $N_0\in \N$ such that the above conditions hold for all $n\geq N_0$.

\begin{Claim}
\label{gen Win}
Let $n\geq N_0$. Fix an \(i\) and let \(a_n \in W_{i,n}\) be such that \(\dim(a_n/K) =
\dim(W_{i,n})\). Then there exists \(b\models p_i\) such that
\(\tp_L(\prol{n}(b)/K) = \tp_L(a_n/K)\).
\end{Claim}

Let \(b \models p_i\) --- a priori we can choose \(b\) in \(M\) but for now we
ignore the \(L\)-structure induced by \(M\) on \(K(\prol{n}(b))\). We have a
field isomorphism \(f_n : K(a_n) \to K(\prol{n}(b))\) sending \(a_n\) to
\(\prol{n}(b)\). By saturation, we can find \(c = (c_i)_{i<\omega} \in M\)
transcendental and algebraically independent over \(K(a_n)\). As, by choice,
\(K(\prol{\omega}(b))\) is purely transcendental over \(K(\prol{n}(b))\), the
isomorphism \(f_n\) extends to a ring isomorphism \(f : K(a_n,c)\to
K(\prol{\omega}(b))\). 

Now \(K(a_n)(c)\) has its \(L\)-structure (as a substructure of M), and the
isomorphism \(f\) induces a new $L$-structure on the differential field
\(K(\nabla_{\omega}(b))\) which is compatible with the ring structure and
extends the \(L\)-structure on K. Let us write  \(K_{1}\) for the differential
field \(K(\nabla_{\omega}(b))\) with this new \(L\)-structure. By construction,
\(f\) is an \(L\)-embedding, and hence the quantifier free $L$-type of
\(\nabla_n(b)\) over \(K\) in $K_{1}$ equals the quantifier free \(L\)-type of
\(a_n\) over \(K\). Also note that \(K_1\) is the expansion of a substructure of
a model of $T$ by $\ell$ commuting derivations containing the model $K$ of
\(\Td\). As \(M\) is a saturated model of \(\Td\) (the model completion of
models of $T$ with \(\ell\) commuting derivations), there is an embedding $h$ of
the \(\Ld\)-structure \(K_{1}\) into \(M\) over \(K\). Let \(u = h(b)\). Then
\(a_{n}\) and \(\prol{n}(u) = h(\prol{n}(b))\) have the same quantifier free
\(L\)-type over \(K\). As \(T\) has quantifier elimination they have the same
\(L\)-type over \(K\) in the model \(M\) of T. This proves \cref{gen Win}.

Now we want to finish the proof of Lemma 2.4. Fix an $L$-formula $\phi(x_{\omega},y)$.
We have to prove that the set of $d\in K$ such that $\phi(x_{\omega},d) \in \Sigma_{X}(x_{\omega})$ is 
$L$-definable over $A$.  

Let $n\geq N_0$ be sufficiently large so that all variables of \(x_\omega\) that
actually appear in \(\phi\) are in \(x_n = \{x_\theta \mid |\theta|\leq n\}\).
We will write \(\phi\) as \(\phi(x_n,y)\).
Now $\phi(x_{n},d)\in \Sigma_{X}(x_{\omega})$ if and only if for all $i=1,..,k$
and $a\in X$ realizing $p_{i}$, we have that $\phi(\nabla_{n}(a), d)$.
Using Claim 2.5 this is equivalent to \(\dim(\text{"\(x_{n}\in Y\cap W_{i,n}\)"}
\wedge \neg\phi(x_{n},d)) < \dim(W_{i,n})\), which by definability of dimension
in $T$ is an $L$-definable condition on $d$.




 It remains to be seen that $\Sigma_{X}$ is \(L\)-definable over $A = \dcld(A)$.
However, note that, for any \(n\), the finite set of (codes of the) \(W_{i,n}\)
is \(\Ld\)-definable over \(A\). By elimination of imaginaries in algebraically
closed fields, it is quantifier free definable in the ring language over $A$.
Hence, \(\Sigma_X\) is \(L\)-definable over \(\dcld(A)\).
\end{proof}

\section{Groups}

For now, let \(T\) be any theory, and  let \(A \subseteq K\models T\). Assume that
\(K\) is \(|A|^+\)-saturated. As previously, we also fix an elementary extension
\(M\supsel K\) which is \(|K|^+\)-saturated in which we realize partial types
over \(K\).

We will use the language of pro-definable sets as in
\cite[Section~2]{HruRK-MetaGp}. A pro-definable set \(X\) over \(A\) is a
(small) projective filtered system \((X_i)_{i\in I}\) of definable sets and
definable maps over \(A\). We think of \(X\) as \(\projlim_{i\in I} X_i\). For
simplicity and in terms of the application, there is no harm in assuming that
the index set \(I\) is countable. As in our applications (to definable groups in
$\Td$) we do not necessarily eliminate imaginaries we distinguish between
pro-definable and pro-interpretable (so pro-interpretable means pro-definable in
$T^{\mathrm{eq}}$). A pro-definable map over \(A\) between pro-definable sets
\(X\) and \(Y = \projlim_{j\in J} Y_j\) is a compatible system of definable maps
\(f_j : X_{i_j} \to Y_j\) over \(A\). Finally, a pro-definable group over \(A\)
is a pro-definable set \(G\) together with a pro-definable group law both over
\(A\). Note that pro-definable sets can alternatively be described as
\(*\)-definable sets, as in \cite[Section~3]{Hru-Unidim}: solutions in $M$ of a
partial type over a small set of parameters in (possibly infinitely) many
variables.

Given a pro-interpretable set \(X = \projlim_{i\in I} X_i\), a (global) partial
type \(\Sigma\) concentrating on \(X\) is a compatible collection of partial
types \(p_i\) over \(K\) (not necessarily over a small set of parameters from
\(K\)) such that \(p_i(x_i)\models x_i\in X_i\). We can see partial types
concentrating on \(X\) as satisfiable collections of pro-definable subsets of
\(X\). We assume that partial types are closed under finite conjunctions and
consequences (these are called "filters" in \cite{HruRK-MetaGp}). For example a
complete type \(p(x)\) over \(K\) implying \(X\) is such a global partial type
concentrating on \(X\). When we talk about a realization of a global partial
type (in maybe infinitely many variables) we mean a realization in $M$, unless
we say otherwise. If $\Sigma$ is a global partial type concentrating on a
pro-definable set \(X\) and \(f: X\to Y\) is pro-definable, the "pushforward"
$f({\Sigma})$ is the global partial type whose set of realizations in $M$ is
precisely $\{f(a): a$ realizes $\Sigma\}$; in other words, for any pro-definable
set \(Z\), we have \(f(\Sigma)(y) \models y\in Z\) if and only if \(\Sigma(x)
\models f(x)\in Z\).

If \(\Sigma(x)\) is a global partial type, we write \(\restr{\Sigma}{A}\) for
its restriction to formulas with parameters in \(A\). Also, as in \cref{gen
def}, we say that \(\Sigma\) is definable over \(A\) if for every formula
\(\phi(x,y)\), the set of tuples \(a\in K^y\) such that
\(\Sigma(x)\models\phi(x,a)\) is \(L\)-definable over \(A\).

\begin{Definition}
Let \(G\) be a pro-interpretable group and \(\Sigma\) a global partial type
concentrating on \(G\). We say that \(\Sigma\) is (left) translation invariant
if for every \(g\in G(K)\) and \(a\models\Sigma\), we have \(g\cdot a
\models\Sigma\) --- equivalently, if \(\Sigma\models X\) and \(g\in G(K)\), then
\(\Sigma\models g\cdot X\).
\end{Definition}

So \(\Sigma\) is translation invariant in \(G\) if and only if \(G = \{g\in
G\mid g\cdot\Sigma = \Sigma\} = \Stab_G(\Sigma)\). Note that a global definable
type concentrating on \(G\) which is translation invariant is a generic
definable filter in the terminology of \cite[Definition~3.1]{HruRK-MetaGp}.
Conversely, if \(G\) admits a generic definable filter then it admits a
definable and translation invariant global partial type by
\cite[Remark~3.3]{HruRK-MetaGp}.

\begin{Definition}
\label{def pregroup}
Let \(\Sigma\) be a global partial type (concentrating on some pro-interpretable
set) which is definable over \(A\) and let \(F(x,y)\) be a map pro-definable
over \(A\). We say that \((\Sigma,F)\) is an abstract group chunk, or pregroup,
over \(A\) if the following hold.
\begin{enumerate}
\item If \(a\models\restr{\Sigma}{A}\) and \(b\models\Sigma\), then \(F(a,b)\) is defined
and \(F(a,\Sigma) = \Sigma\).
\item If \(a\models\restr{\Sigma}{A}\) and \(b\models\restr{\Sigma}{Aa}\), then \(a\) and \(b\)
are interdefinable over \(A\cup\{F(a,b)\}\).
\item If \(a\models\restr{\Sigma}{A}\), if \(b\models\restr{\Sigma}{Aa}\) and if
\(c\models\restr{\Sigma}{Aab}\), then \(F(a,F(b,c)) = F(F(a,b),c)\).
\end{enumerate}
\end{Definition}

These are the main results we will use on these notions.

\begin{Proposition}[{\cite[Prop.~3.15]{HruRK-MetaGp}}]
\label{pregroup}
Let \((\Sigma,F)\) be a pregroup over \(A\). Then there exists a
pro-interpretable group \(G\) over \(A\) and an injective map \(f :
\restr{\Sigma}{A}\to G\) which is pro-interpretable over \(A\) and such that:
\begin{enumerate}
\item for any \(a\models\restr{\Sigma}{A}\) and \(b\models\restr{\Sigma}{Aa}\),
\(f(F(a,b)) = f(a)\cdot f(b)\);
\item the global partial type \(f(\Sigma)\) is translation invariant in \(G\).
\end{enumerate}
\end{Proposition}

\begin{Proposition}[{\cite[Prop.~3.4]{HruRK-MetaGp}}]
\label{*group prolim}
Let \(G\) be a pro-interpretable group over \(A\) and let \(\Sigma\) be a
translation invariant global partial type concentrating on \(G\) definable over
\(A\). Then \(G\) is pro-interpretably over \(A\) isomorphic to a projective
limit of groups interpretable over \(A\).
\end{Proposition}
    
\begin{Proposition}[{\cite[Prop.~3.16]{HruRK-MetaGp}}]
\label{gen morphism}
Let \(G\) and \(H\) be pro-interpretable groups over \(A\), let \(\Sigma\) be a
translation invariant partial type concentrating on \(G\) definable over \(A\)
and let \(f\) be a pro-definable map over \(A\) such that for every
\(a\models\restr{\Sigma}{A}\), \(f(a) \in H\). Assume moreover, that for all
\(a\models\restr{\Sigma}{A}\) and \(b\models\restr{\Sigma}{Aa}\), \(f(a\cdot b)
= f(a)\cdot f(b)\). Then there exists a unique pro-interpretable (over \(A\))
group morphism \(g : G \to H\) agreeing with \(f\) on realizations of
\(\Sigma\).
\end{Proposition}

In other words, there is an equivalence of categories between pregroups,
pro-interpretable groups with a translation invariant definable partial type, and
projective limits of interpretable groups with a translation invariant definable
partial type. In particular, in \cref{gen morphism}, the map \(g\) is injective
if and only if \(f\) is.

Now, let \(T\) be a geometric theory of enriched characteristic zero fields.

\begin{Theorem}
\label{emb Gp}
Let $K \models\Td$, let \(A = \dcld(A)\subseteq K\) and let $\Gamma$
be a group \(\Ld\)-definable in $K$ over \(A\). Then there is a group
\(G\) which is \(L\)-interpretable over \(A\) and a group embedding
\(\Gamma \to G\) which is \(\Ld\)-definable in \(K\) over \(A\).
\end{Theorem}

\begin{proof}
Let \(\Sigma(x_\omega)\) be the global partial \(L\)-type such that
\(\Sigma(\prol{\omega}(x))\) is the partial type of generics in \(\Gamma\). By
\cref{gen def}, it is \(L\)-definable over \(A\). By \cref{dcl Td}.3, there
exists a map \(F\) which is pro-definable in \(L\) over \(A\) such that for
every \(a,b\in\Gamma\), we have \(\prol{\omega}(a\cdot b) =
F(\prol{\omega}(a),\prol{\omega}(b))\). Likewise, there are functions \(G_1\) and
\(G_2\) which are pro-definable in \(L\) over \(A\) such that for any \(a, b
\in \Gamma\), we have \(\prol{\omega}(a\cdot b^{-1}) =
G_1(\prol{\omega}(a),\prol{\omega}(b))\) and \(\prol{\omega}(a^{-1}\cdot b) =
G_2(\prol{\omega}(a),\prol{\omega}(b))\).

\begin{Claim}
\((\Sigma,F)\) is a pregroup.
\end{Claim}

\begin{proof}
First fix \(a\in\Gamma(K)\) and \(b\in\Gamma(M)\). By \cref{inj generic},
\(a\cdot b\) is generic in \(\Gamma\) over \(K\) if and only if \(b\) is. Namely,
\(\prol{\omega}(b)\) realizes \(\Sigma\) if and only if \(\prol{\omega}(a\cdot
b) = F(\prol{\omega}(a),\prol{\omega}(b))\) also does. As \(\Sigma(x_\omega)\)
is the \(L\)-type of all tuples \(\prol{\omega}(b)\) for \(b\in \Gamma\) generic
over \(K\), it follows that for any \(b_\omega\models\Sigma\), we have
\(F(\prol{\omega}(a),b_\omega)\models\Sigma\) and moreover, that every
\(c_\omega\models\Sigma\) is of the form \(F(\prol{\omega}(a),b_\omega)\) for
some \(b_\omega\models\Sigma\). So we have \(F(\prol{\omega}(a),\Sigma) =
\Sigma\). 

By definability of \(\Sigma\), the set of tuples \(a_\omega\in K\) such that
\(F(a_\omega,\Sigma) = \Sigma\) is pro-definable over \(A\). As it includes
\(\prol{\omega}(a)\) for any \(a\in\Gamma\), it also includes all realizations
of \(\restr{\Sigma}{A}\). This yields condition 1 in \cref{def pregroup}.

Condition 2 and 3 hold for similar reasons. For example, let us consider
condition 2. Again, fix \(a\in\Gamma(K)\) and \(b\in\Gamma(M)\) and let \(c =
a\cdot b\). Then \(\prol{\omega}(c) = F(\prol{\omega}(a),\prol{\omega}(b))\). We
also have \(\prol{\omega}(a) = G_1(\prol{\omega}(c),\prol{\omega}(b))\) and
\(\prol{\omega}(b) = G_2(\prol{\omega}(a),\prol{\omega}(c))\). So, by definition
of \(\Sigma\), for every \(b_\omega\models\Sigma\), we have \(\prol{\omega}(a) =
G_1(F(\prol{\omega}(a),b_\omega),b_\omega)\) and \(b_\omega =
G_2(\prol{\omega}(a),F(\prol{\omega}(a),b_\omega))\). Again, by definability of
\(\Sigma\) over \(A\), the set of tuples \(a_\omega\) such that \(a_\omega =
G_1(F(a_\omega,b_\omega),b_\omega) = G_2(b_\omega,F(a_\omega,b_\omega))\) is
pro-definable in \(L\) over \(A\), and it contains \(\prol{\omega}(a)\) for all
\(a\in\Gamma(K)\). In particular, it contains all realizations of
\(\restr{\Sigma}{A}\). This yields condition 2 in \cref{def pregroup}.
\end{proof}

Let us now come back to the proof of the theorem. By Propositions \ref{pregroup}
and \ref{*group prolim}, we obtain a projective limit \(G = \prolim_i G_i\) of
groups which are \(L\)-interpretable over \(A\), as well as an injective map
\(f\) from realizations of \(\restr{\Sigma}{A}\) to \(G\) (in \(M\)) which is
pro-definable in \(L\) over \(A\) and such that for every
\(a\models\restr{\Sigma}{A}\) and \(b\models\restr{\Sigma}{Aa}\), we have
\(f(F(a,b)) = f(a)\cdot f(b)\).

Note that \(\Sigma(\prol{\omega}(x))\) is the global definable over \(A\)
translation invariant \(\Ld\)-type of generics of \(\Gamma\) over \(K\). We can
therefore apply \cref{gen morphism} to the map \(f\circ\prol{\omega}\) from
realizations of \(\restr{\Sigma(\prol{\omega}(x))}{A}\) to \(G\) to obtain a
group embedding \(g : \Gamma \to G = \prolim_i G_i\) which is pro-definable over
\(A\). By compactness, the composition of \(g\) with the projection on some
\(G_i\) is already injective and this completes the proof.
\end{proof}

\sloppy
\printbibliography

\end{document}